\documentclass[reqno,11pt]{amsart}
\usepackage[margin=1in]{geometry}

\usepackage{amsthm, amsmath, amssymb, bm, bbm}
\usepackage{enumerate}
\usepackage{tikz}
\usepackage{makecell}

\usepackage[utf8]{inputenc}
\usepackage[T1]{fontenc}

\usepackage[textsize=scriptsize,backgroundcolor=orange!5]{todonotes}

\usepackage[hidelinks]{hyperref}
\usepackage{url}

\usepackage[noabbrev,capitalize]{cleveref}
\crefname{equation}{}{}

\usepackage{xcolor,graphics}

\numberwithin{equation}{section}

\newtheorem{theorem}{Theorem}[section]

\newtheorem{lemma}[theorem]{Lemma}

\newtheorem*{question*}{Question} \Crefname{question}{Question}{Questions}

\newtheorem*{theorem*}{Theorem}

\theoremstyle{definition}

\newtheorem{example}[theorem]{Example}

\theoremstyle{remark}

\newcommand{\nc}{\newcommand}

\newcommand{\one}{\mathbbm{1}}

\nc{\on}{\operatorname}
\nc{\mf}{\mathfrak}
\nc{\mc}{\mathcal}

\newcommand{\FF}{\mathbb{F}}
\newcommand{\QQ}{\mathbb{Q}}

\newcommand{\NN}{\mathbb{N}}

\newcommand{\eps}{\varepsilon}

\title{M\"obius Formulas for Densities of Sets of Prime Ideals}

\author[Kural]{Michael Kural}
\address{Massachusetts Institute of Technology, Cambridge, MA 02139, USA}
\email{mkural@mit.edu}

\author[McDonald]{Vaughan McDonald}
\address{Harvard University, Cambridge, MA 02138, USA}
\email{vmcdonald@college.harvard.edu}

\author[Sah]{Ashwin Sah}
\address{Massachusetts Institute of Technology, Cambridge, MA 02139, USA}
\email{asah@mit.edu}

\date{\today}

\allowdisplaybreaks

\begin{document}

\maketitle

\begin{abstract}
We generalize results of Alladi, Dawsey, and Sweeting and Woo for Chebotarev densities to general densities of sets of primes. We show that if $K$ is a number field and $S$ is any set of prime ideals with natural density $\delta(S)$ within the primes, then
\[
-\lim_{X \to \infty}\sum_{\substack{2 \le \on{N}(\mf{a})\le X\\ \mf{a} \in D(K,S)}}\frac{\mu(\mf{a})}{\on{N}(\mf{a})} = \delta(S),
\]
where $\mu(\mf{a})$ is the generalized M\"obius function and $D(K,S)$ is the set of integral ideals $ \mf{a} \subseteq \mc{O}_K$ with unique prime divisor of minimal norm lying in $S$. Our result can be applied to give formulas for densities of various sets of prime numbers, including those lying in a Sato-Tate interval of a fixed elliptic curve, and those in Beatty sequences such as $\lfloor\pi n\rfloor$.
\end{abstract}

\section{Introduction and Statement of Results}\label{sec:introduction}

It is a well-known fact that the identity
\[ 1 = -\sum_{n\ge 2}\frac{\mu(n)}{n} = \frac{1}{2} + \frac{1}{3} + \frac{1}{5} - \frac{1}{6} + \frac{1}{7} - \cdots,\]
where $\mu(n)$ denotes the M\"obius function, is equivalent to the classical prime number theorem, which states $\pi(x) \sim \on{Li}(x)$. This paper explores how general analogues of the prime number theorem can give similar identities involving the M\"obius function. First, Alladi \cite{Al77} recognized that such a formula could be directly interpreted as a statement about the density of primes. In particular, let $p_{\min}(n)$ denote the smallest prime divisor of an integer $n \ge 2$. Alladi \cite{Al77} proved
\[
- \sum_{\substack{n\ge 2\\ p_{\min}(n) \equiv a \pmod{q}}} \frac{\mu(n)}{n} = \frac{1}{\varphi(q)}
\]
for any positive integer $q$ and integer $a$ with $\gcd(a, q) =1$.

Dawsey \cite{Da17} extended Alladi's result to the setting of Chebotarev densities for finite Galois extensions of $\QQ$, while Sweeting and Woo \cite{SW19} generalized this to finite Galois extensions of number fields $L/K$. In a somewhat different direction, Ska\l ba \cite{Sk95} and Tao \cite{T10} considered sums of $\mu(n)/n$ with all prime factors of $n$ restricted to a fixed subset of the primes (and analogous sums for ideals in number fields). These results use some related techniques but recover different identities; for example, in the cases of arithmetic progressions and Chebotarev sets, such a sum converges to $0$.

We wish to further generalize the results of Alladi, Dawsey, and Sweeting and Woo to arbitrary sets of prime ideals in a number field $K$. To do so, we define the concept of the \emph{natural density} of $S\subseteq\mc{P}$, where $\mc{P}$ denotes the set of integral prime ideals of $K$:
\begin{equation}
\delta(S) := \lim_{X\to\infty}\frac{\pi_S(K; X)}{\pi(K; X)},
\end{equation}
where $\pi_S(K; X) = \#\{\mf{p}\in S:\on{N}(\mf{p})\le X\}$ and $\pi(K; X) = \pi_{\mc{P}}(K; X)$. Note that this density does not necessarily exist for arbitrary sets $S$.

We require some terminology to state our main result. An integral ideal $\mf{a}$ is said to be \emph{distinguishable}\footnote{Sweeting and Woo \cite{SW19} call these ideals \emph{salient}.} if there is a unique prime ideal $\mf{p}\supseteq\mf{a}$ attaining the minimum norm of all such primes. For $\mf{a}$ distinguishable, let $\mf{p}_{\min}(\mf{a})$ denote this minimal prime. We define 
\begin{equation}D(K, S) := \{\mf{a}\subseteq\mc{O}_K\text{ is distinguishable}: \mf{p}_{\min}(\mf{a}) \in S\}.\end{equation}

We will also require the following natural generalization of the M\"obius function to ideals $\mf{a}\subseteq\mc{O}_K$:
\begin{equation}
\mu_K(\mf{a}) :=
\begin{cases}
1 & \text{ if }\mf{a} = \mc{O}_K,\\
0 & \text{ if }\mf{a} \supseteq \mf{p}^2 \text{ for }\mf{p} \text{ prime},\\
(-1)^k & \text{ if }\mf{a} = \mf{p}_1 \dots \mf{p}_k.
\end{cases}
\end{equation}
Throughout we will write $\mu(\mf{a}) := \mu_K(\mf{a})$ when the context is clear.

Our general result is as follows:

\begin{theorem}\label{thm:main-result}
Fix a number field $K$. If $S\subseteq\mc{P}$ has a natural density $\delta(S)$, then we have that
\[-\lim_{X\to\infty}\sum_{\substack{2\le\on{N}(\mf{a})\le X\\\mf{a}\in D(K,S)}}\frac{\mu(\mf{a})}{\on{N}(\mf{a})} = \delta(S).\]
\end{theorem}
To obtain this result, we generalize the method of Alladi \cite{Al77}, relying in particular on a certain ``duality property'' of minimal and maximal prime divisors. In the case $K = \QQ$, this property takes the form
\begin{equation}\label{eq:alladidualityIntegers}
\sum_{d\mid n}\mu(d) f(p_{\max}(d)) = -f(p_{\min}(n))
\end{equation}
for $\mu = \mu_{\QQ}$, $p_{\max}(n)$ the maximum prime divisor of $n$, and $f$ a function on integers with $f(1)=0$. The duality principle along with methods from analytic number theory are the key ingredients necessary to prove \cref{thm:main-result}.

In \cref{sec:definitions}, we define some basic notions and essential quantities that we will bound throughout the paper. In \cref{sec:necessary-lemmata}, we recall and prove key number- and ideal-theoretic estimates involving these quantities. In \cref{sec:key-step}, we establish the intermediate sum \cref{thm:key-step}. In \cref{sec:a-mobius-sum}, we use a duality argument along with the more general version of Axer's theorem \cite{Ax10} to establish \cref{thm:main-result}.

Finally, in \cref{sec:applications} we discuss several interesting applications and new examples arising from these results. This includes formulas for densities of primes in Sato-Tate intervals for elliptic curves without CM as well as densities of primes in Beatty sequences. As an example, we will be able to conclude
\[
- \sum_{\substack{n \ge 2 \\ p_{\min}(n) \in \mc{B}_\pi}} \frac{\mu(n)}{n} = \frac{1}{\pi},
\]
where $\mc{B}_\pi = \{3,31,\ldots\}$ denotes the set of primes of the form $\lfloor \pi n \rfloor$ for $n \in \NN$.

\section{Nuts and Bolts}\label{sec:definitions}
In this section we lay out some necessary terminology and lemmata that will be used throughout the proof.
\subsection{Notation and definitions}\label{sec:notation-definitions}
Fix a number field $K$ and let $\mc{O}_K$ be its ring of integers. Denote its set of prime ideals by $\mc{P}$. Now, for the rest of our discussion, we fix a subset $S\subseteq\mc{P}$ of prime ideals that has a natural density. 

By the prime ideal theorem \cite{L03}, we have $\pi(K;X)\sim\on{Li}(X)$; thus if $S$ has a natural density, we analogously have $\pi_{S}(K;X) \sim \delta(S)\on{Li}(X)$. We define the error on such an estimate to be 
\begin{equation}
e_S(X) := \sup_{Y\le X}|\pi_S(K;Y) - \delta(S)\on{Li}(Y)|,
\end{equation}
which satisfies $e_S(X)/X = o(1/\log X)$ by the previous asymptotic. To transform this discrepancy into a monotonically decreasing function, define
\begin{equation}
v_S(X) := \sup_{Y\ge X}\frac{e_S(Y)}{Y}.
\end{equation}
In particular, note that for all $X$, $e_S(X)\le Xv_S(X)$.

We will also use estimates on smooth ideals for intermediate purposes. By smooth we mean ideals that have only prime factors of small norm. These estimates will be related to the ``maximal'' primes dividing an given ideal. For an ideal $\mf{a}$, we define its maximal prime norm
\begin{equation}\label{eq: maximal norm primes}
M(\mf{a}) := \max_{\mf{p}\supseteq \mf{a}}\on{N}(\mf{p})
\end{equation}
and let 
\begin{equation}\label{eq:maximalprimenorm}
Q_S(\mf{a}) := \#\{\mf{p}\supseteq \mf{a}: \on{N}(\mf{p}) = M(\mf{a}), \mf{p}\in S\}.
\end{equation}
Then take
\begin{equation}
\mc{S}(X,Y) := \{\mf{a}\subseteq\mc{O}_K: \on{N}(\mf{a})\le X, M(\mf{a}) \le Y\},
\end{equation}
which is the set of $Y$-smooth ideals of norm at most $X$, and let
\begin{equation}
\Psi(X, Y) := \#\mc{S}(X, Y).
\end{equation}

\subsection{Preliminary lemmata}\label{sec:necessary-lemmata}
To discuss the density of prime ideals, the first theorem we will need is an effective form of the prime ideal theorem (see \cite{LO77}).
\begin{theorem}
If $K$ is a number field, then there exists constants $C_1,C_2 > 0$ such that, for large enough $X$, we have
\[
|\pi(K;X) - \on{Li}(X)| \le C_1X\exp(-C_2\sqrt{\log X}).
\]
\end{theorem}
We will require an analogue of Alladi's duality identity \eqref{eq:alladidualityIntegers} to our context of ideals:
\begin{lemma}\label{lem:mobius-inversion}
For a fixed $\mf{a}$, we have
\[
\sum_{\mf{b}\supseteq\mf{a}} \mu(\mf{b}) \one_D(\mf{b}) = -Q_S(\mf{a}),
\]
where $\one_{D}$ denotes the indicator function for $D(K,S)$.
\end{lemma}
\begin{proof}
This directly generalizes \cite[Lemma~2.1]{SW19} and the proof follows \emph{mutatis mutandis}.
\end{proof}
The identity will convert our discussion from the domain of minimal prime divisors to that of maximal prime divisors. As such, we will need use the following estimate on smooth ideals.
\begin{theorem}[{\cite[Lemma~4.1]{M92}}]\label{thm:psi-bound}
Fix $\eps > 0$. If $1\le\beta:=\beta(X)\le (\log X)^{1 - \eps}$ then
\[\Psi(X, Y) = O_\eps\left(X\exp\left(-\beta\log\frac{\beta}{2}\right)\right).\]
\end{theorem}
The result follows from \cite[Lemma~4.1]{M92}, after using that $\rho(\beta)\le 1/\Gamma(\beta+1)$ and Stirling's formula, where $\rho$ is the Dickman function satisfying
\begin{align*}
\rho(\beta) &= 1\text{ if }\beta\in [0, 1],\\
-\beta\rho'(\beta) &= \rho(\beta - 1)\text{ if }\beta > 1.
\end{align*}
By \cite[Lemma~4.1]{M92}, it is known that
\[\psi(X, Y) = c_KX\rho(\beta)\left(1 + O_\eps\left(\frac{\beta\log(\beta+1)}{\log X}\right)\right).\]

To compute certain integral ideal sums, we will need some knowledge of the structure of the number field $K$. Let $\zeta_K(s)$ be the Dedekind zeta function for the number field $K$, and let its residue at $s = 1$ be
\begin{equation}
c_K = \lim_{s\to 1}(s - 1)\zeta_K(s) = \frac{2^{r_1}(2\pi)^{r_2}\on{Reg}_Kh_K}{w_K\sqrt{|D_K|}}.
\end{equation}
This is the analytic class number formula: here $r_1$ is the number of real embeddings and $r_2$ the number of conjugate pairs of complex embeddings of $K$, $h_K$ is the class number, $w_K$ is the number of roots of unity in $K$, $D_K$ is the absolute discriminant, and $\on{Reg}_K$ is the regulator. We now have the following estimates, which will be vital in the final steps of the proof of \cref{thm:main-result}. For convenience, define $Q(\mf{a}) := Q_{\mc{P}}(\mf{a})$.
\begin{lemma}[{\cite[Lemma~3.3]{SW19}, \cite[Lemma~2.2]{Sh49}}]\label{lem:mu-sums-number-field}
There is a constant $C = C(K) > 0$ such that if $X\ge 1$, then
\begin{align}
\sum_{2\le \on{N}(\mf{a})\le X} Q(\mf{a}) &= c_KX + O\left(X\exp\left(-C(\log X)^{1/3}\right)\right),\label{eq:ideal-sum-q}\\
\sum_{\on{N}(\mf{a})\le X} \frac{1}{\on{N}(\mf{a})} &= c_K \log X + O(1).\label{eq:ideal-sum-n}
\end{align}
\end{lemma}

\section{Smooth Ideal Sums}\label{sec:key-step}
In order to arrive at our final identity \cref{thm:main-result}, we will need a key intermediate estimate of a sum involving $Q_S(\mf{a})$ from \eqref{eq: maximal norm primes}:
\begin{theorem}\label{thm:key-step}
If $S$ has natural density within the prime ideals $\delta(S)$, then we have the asymptotic
\[\sum_{2\le\on{N}(\mf{a})\le X}Q_S(\mf{a})\sim c_K\delta(S)X.\]
\end{theorem}
\begin{proof}
We define a threshold $Y(X) := X^{1/\beta}$ to be chosen later, where $\beta$ satisfies the following properties:
\begin{enumerate}[i.]
\item $\beta\ge 1$ is a monotonically increasing function of $X$ such that $\lim_{X\to \infty} \beta = \infty$.
\item $\beta\le \sqrt{\log X}$.
\end{enumerate}
Note that necessarily $\lim_{X \to \infty} Y(X) = \infty$.

First we write the sum in terms of $\Psi(X,Y)$ via counting the prime divisors of $\mf{a}$ with largest norm:
\begin{align*}
\sum_{2 \le\on{N}(\mf{a}) \le X}Q_S(\mf{a}) = \sum_{\substack{\on{N}(\mf{p}) \le X\\\mf{p}\in S}}\Psi\left(\frac{X}{\on{N}(\mf{p})}, \mf{p}\right).
\end{align*}
We break up the sum into
\begin{align*}
S_1:= \sum_{\substack{\on{N}(\mf{p})\le Y\\\mf{p} \in S}}\Psi\left(\frac{X}{\on{N}(\mf{p})}, \on{N}(\mf{p})\right),\qquad S_2:= \sum_{\substack{Y < \on{N}(\mf{p}) \le X\\\mf{p} \in S}}\Psi\left(\frac{X}{\on{N}(\mf{p})}, \on{N}(\mf{p})\right). 
\end{align*}
We note that 
\[
S_1 \le \sum_{\on{N}(\mf{p}) \le Y}\Psi\left(\frac{X}{\on{N}(\mf{p})}, \on{N}(\mf{p})\right) \le  [K: \QQ]\Psi(X,Y).
\]
Recall $\beta = (\log X)/(\log Y)$. Since we are choosing $\beta\le\sqrt{\log X}$, it follows from \cref{thm:psi-bound} that
\[
S_1 = O\left(X\exp\left(-\beta\log\frac{\beta}{2}\right)\right).
\]
In particular, $S_1$ has sublinear growth provided that $\beta\to\infty$ as $X \to \infty$. Now, to estimate $S_2$, we first extract what will end up being the main term. Namely, let
\[
S_3 = \sum_{\substack{Y < \on{N}(\mf{p}) \le X\\\mf{p} \in S}}\Psi\left(\frac{X}{\on{N}(\mf{p})},\on{N}(\mf{p})\right) - \delta(S)\int_Y^X\Psi\left(\frac{X}{t},t\right)\frac{dt}{\log t}.
\]
Now, utilizing the definition of $\Psi$ this sum turns into
\begin{align*}
S_3 &= \sum_{\substack{Y\le\on{N}(\mf{p})\le X\\\mf{p}\in S}}\sum_{\mf{a}\in\mc{S}\left(X/\on{N}(\mf{p}),\on{N}(\mf{p})\right)}1 - \delta(S)\int_Y^X\left(\sum_{\mf{a}\in\mc{S}\left(X/\on{N}(\mf{p}),\on{N}(\mf{p})\right)}1\right)\frac{dt}{\log t}\\
&= \sum_{\substack{1\le\on{N}(\mf{a})\le X/Y\\M(\mf{a})\le X/\on{N}(\mf{a})}}\Bigg(\sum_{\substack{M(\mf{a})\le\on{N}(\mf{p})\le X/\on{N}(\mf{a})\\\on{N}(\mf{p})>Y\\\mf{p}\in S}}1 - \delta(S)\int_{\max(M(\mf{a}), Y)}^{X/\on{N}(\mf{a})}\frac{dt}{\log t}\Bigg)\\
&= \sum_{\substack{1\le\on{N}(\mf{a})\le X/Y\\M(\mf{a})\le X/\on{N}(\mf{a})}}\bigg(\pi_S\left(K; \frac{X}{\on{N}(\mf{a})}\right) - \pi_C\left(K;\max(M(\mf{a}), Y)\right)\\
&\qquad\qquad\qquad\qquad-\delta(S)\on{Li}\left(\frac{X}{\on{N}(\mf{a})}\right) + \delta(S)\on{Li}(\max(M(\mf{a}), Y))\bigg),
\end{align*}
switching the sum and integral to obtain the second line. Now by definition the error function $e_S(X)$ is monotonically increasing, so 
\begin{align*}
|S_3| &\le 2\sum_{\substack{1 \le \on{N}(\mf{a}) \le X/Y\\M(\mf{a}) \le X/\on{N}(\mf{a})}}e_S(X/\on{N}(\mf{a}))\le 2X\sum_{\substack{1 \le \on{N}(\mf{a}) \le X/Y\\M(\mf{a}) \le X/\on{N}(\mf{a})}}\frac{1}{\on{N}(\mf{a})}\cdot v_S(X/\on{N}(\mf{a}))\\
&\ll X\cdot \log(X/Y) \cdot v_S(Y)\ll Xv_S(Y)\log X.
\end{align*}
where we use that $v_S(X)$ is monotonically decreasing. Now, since we want this sum to be sublinear, we will later choose $\beta$ such that $v_S(X^{1/\beta})\log X\to 0$ as $X\to\infty$. Overall, we have shown
\begin{equation}\label{eq:we-have-shown}
\sum_{2 \le\on{N}(\mf{a}) \le X}Q_S(\mf{a}) = \delta(S)\int_Y^X\Psi\left(\frac{X}{t}, t\right)\frac{dt}{\log t} +O\left(X\exp\left(-\beta\log\frac{\beta}{2}\right)\right) + O(Xv_S(Y)\log X).
\end{equation}
Applying \cref{eq:we-have-shown} to the set $\mc{P}$ of all prime ideals, we have
\[\sum_{2 \le\on{N}(\mf{a}) \le X}Q(\mf{a}) = \int_Y^X\Psi\left(\frac{X}{t},t\right)\frac{dt}{\log t} + O\left(X\exp\left(-\beta\log\frac{\beta}{2}\right)\right) + O(Xw(Y)\log X),\]
where
\[w(x):=\sup_{y\ge x}\frac{\sup_{z\le y}|\pi_{\mc{P}}(K; z) - \on{Li}(z)|}{y}\]
is the normalized error in the prime ideal theorem. Using \cref{eq:ideal-sum-q} of \cref{lem:mu-sums-number-field} and combining these two equations, we thus have
\begin{align*}
\sum_{2\le\on{N}(\mf{a})\le X}Q_S(\mf{a}) &= c_K\delta(S)X + O\left(X\exp\left(-C(\log X)^{1/3}\right)\right) + O\left(X\exp\left(-\beta\log\frac{\beta}{2}\right)\right)\\
&\qquad\qquad\qquad\qquad\qquad\qquad+ O(X(w(Y) + v_S(Y))\log X).
\end{align*}
Let $u(x) = w(x) + v_S(x)$ for convenience. Note that $u(x)$ also monotonically decreases. Now we are ready to choose $\beta$. First, consider any $\eps > 0$. Then there exists a minimum positive integer constant $C = C(\eps)$ such that $u(X^\eps)\log X < \eps$ for all $X > C(\eps)$ by the prime ideal theorem \cite{LO77} and the fact that the natural density of $S$ exists. In particular, we are using that $u(x) = o(1/\log x)$.

We can now set
\[
\beta = \min\left(\left\lfloor\sqrt{\log X}\right\rfloor, \sup\left\{m\in\NN: C\left(\frac{1}{m}\right) < X\right\}\right).
\]
The motivation for this formula is that we increase $\beta$ according to a discrete set of thresholds $C(1/m)$ that make the desired expression $u(X^{1/\beta})\log X$ small. Note in particular that the set within the supremum is a downwards-closed subset of the positive integers by monotonicity of $C(\eps)$.

Note that $C(\eps)$ increases as $\eps$ decreases. Additionally, $\beta \to \infty$ as $X\to\infty$ and this function is nondecreasing. Also, we observe that
\[u\left(X^{1/\beta}\right)\log X\le\frac{1}{\beta}\]
by construction and the monotonicity of the constants $C(\eps)$. Indeed, if $\beta = m$ then we see that $C(1/m) < X$ and thus $u(X^{1/m})\log X < 1/m$ by definition. And now the above discussion immediately implies that
\[\lim_{X\to\infty}u\left(X^{1/\beta}\right)\log X = 0.\]
Thus, putting it all together, we have
\[\sum_{2\le\on{N}(\mf{a})\le X}Q_S(\mf{a}) = c_K\delta(S)X + o(X),\]
as desired.
\end{proof}

\section{A M\"obius Sum for Prime Ideals}\label{sec:a-mobius-sum}
In this section, we use \cref{thm:key-step} to prove our main result, \cref{thm:main-result}. This conversion is a generalization of a theorem of Alladi \cite[Theorem~6]{Al77} to the context of prime ideals. To achieve this conversion, we will need more general versions of some of Alladi's tools, including Axer's theorem.
\begin{theorem}[Axer {\cite{Ax10}}]\label{thm:axer}
Define $[X]_K = |\{\mf{a} \subseteq \mc{O}_K: \on{N}(\mf{a}) \le X\}|$.
If $f(\mf{a})$ is an ideal-theoretic function satisfying
\begin{align*}
\sum_{\on{N}(\mf{a})\le X} |f(\mf{a})| =O(X),\qquad\sum_{\on{N}(\mf{a})\le X} f(\mf{a})= o(X),
\end{align*}
then 
\[
\sum_{\on{N}(\mf{a})\le X}\left(c_K\frac{X}{\on{N}(\mf{a})} - \left[\frac{X}{\on{N}(\mf{a})}\right]_K\right)f(\mf{a}) = o(X).
\]
\end{theorem}
In order to verify the conditions of Axer's theorem when we apply it, we need the following intermediate bound.
\begin{lemma}\label{lem:mu-D-sum}
We have 
\[\sum_{\substack{\on{N}(\mf{a})\le X}}\mu(\mf{a})\one_{D}(\mf{a}) = o(X).\]
\end{lemma}
\begin{proof}
We break up the sum based on the unique minimal prime divisor of $\mf{a}$, since $D(K, S)$ only contains distinguishable ideals. For any $N\ge 1$, let $\mf{c}_N$ be the product of all prime ideals of absolute norm at most $N$.
\begin{equation}\label{eq:alladi-theorem-5}
\begin{split}
\left|\sum_{\substack{\on{N}(\mf{a})\le X}}\mu(\mf{a})\one_{D}(\mf{a})\right| &= \left|\sum_{\substack{\on{N}(\mf{p}) \le X\\\mf{p}\in S}}\sum_{\substack{\on{N}(\mf{b})\le X/\on{N}(\mf{p})\\(\mf{b},\mf{c}_{\on{N}(\mf{p})})=(1)}}\mu(\mf{p}\mf{b})\right|\\
&\ll\sum_{N(\mf{p})\le\alpha(X)}\left|\sum_{\substack{\on{N}(\mf{b})\le X/\on{N}(\mf{p})\\(\mf{b},c_{\on{N}(\mf{p})})=(1)}}\mu(\mf{b})\right| + \sum_{\alpha(X) < \on{N}(\mf{p})\le X}\sum_{\substack{\on{N}(\mf{b})\le X/\on{N}(\mf{p})\\(\mf{b},c_{\on{N}(\mf{p})})=(1)}} 1 =: S_4 + S_5
\end{split}
\end{equation}
where $\alpha(X)$ is a slow-growing function of $X$ to be chosen later. We thus will derive estimates on the following two sums:
\[
M_{\mf{c}}(Y):= \sum_{\substack{\on{N}(\mf{a}) \le Y \\ (\mf{a}, \mf{c}) = (1)}}\mu(\mf{a}),\qquad \phi(\mf{c},Y):= \sum_{\substack{\on{N}(\mf{a}) \le Y\\ (\mf{a}, \mf{c}) = (1)}}1.
\]
By standard Mellin inversion techniques, we see that
\begin{align*}
M_\mf{c}(Y) = \int_{2-i\infty}^{2+i\infty}\zeta_K(s)^{-1}\prod_{\mf{p}\supseteq\mf{c}}(1-\on{N}(\mf{p})^{-s})^{-1}\frac{Y^s}{s}\,ds.
\end{align*}
Estimating this is a matter of pushing the contour to a zero-free region; this standard computation is done explicitly for example in \cite[Theorem~1.2]{FM12}. Following this argument, we can derive
\[M_\mf{c}(Y)\prod_{\mf{p}\supseteq\mf{c}}(1-\on{N}(\mf{p})^{-1/2})\ll g(Y) = o(Y),\]
where the function $g(Y)$ side is uniform in $\mf{c}$. By the principle of inclusion-exclusion, we find
\[\phi(\mf{c}, Y) = \sum_{\mf{a}' \supseteq \mf{c}} \mu(\mf{a}') \sum_{\substack{\on{N}(\mf{a})\le Y\\\mf{a}'\supseteq \mf{a}}} 1=c_KY\prod_{\mf{p}\supseteq\mf{c}}(1-\on{N}(\mf{p})^{-1}) + O(Y^{1-1/[K:\QQ]}\cdot\on{N}(\mf{c})),\]
using the classical bound $[X]_K = c_KX + O(X^{1-1/[K:\QQ]})$ to control the error term as well as the fact that at most $\on{N}(\mf{c})$ ideals divide $\mf{c}$. Thus we have
\begin{align*}
S_4&\ll\pi_K(\alpha(X))\cdot\prod_{\on{N}(\mf{p})\le\alpha(X)}(1-\on{N}(\mf{p})^{-1/2})^{-1}\cdot g(X),\\
S_5&\ll M_{\mf{c}_{\alpha(X)}}(X) = c_KX\prod_{\on{N}(\mf{p})\le\alpha(X)}(1-\on{N}(\mf{p})^{-1}) + O\left(X^{1-1/[K:\QQ]}\prod_{\on{N}(\mf{p})\le\alpha(X)}\on{N}(\mf{p})\right).
\end{align*}
It is immediately apparent that taking $\alpha(X)$ to be a slow enough growing unbounded monotonic function (in terms of the error term $g(X)$ and so that the bound on $S_5$ is $o(X)$) gives $S_4 = o(X)$ and $S_5 = o(X)$. Here we used the fact that
\[\prod_{\on{N}(\mf{p})\le\alpha(X)}(1-\on{N}(\mf{p})^{-1})\le\left(\sum_{\on{N}(\mf{a})\le\alpha(X)}\on{N}(\mf{a})^{-1}\right)^{-1}\ll\frac{1}{\log\alpha(X)}.\]
Now, plugging into \cref{eq:alladi-theorem-5} finishes. In fact, $g(Y) = Y\exp(-A(\log Y)^{-1/2})$ for some $A > 0$ depending on $K$ and we can take $\alpha(X) = (\log X)^{1/2}$. This actually gives the effective bound
\[\sum_{\on{N}(\mf{a})\le X}\mu(\mf{a})\one_D(\mf{a})\ll\frac{X}{\log\log X},\]
using the prime ideal theorem and summation by parts repeatedly.
\end{proof}
This allows us to prove the following conversion.
\begin{lemma}\label{lem:convert-thm}
If we have
\[\sum_{2\le\on{N}(\mf{a})\le X} Q_S(\mf{a}) = c_K\delta(S)X + o(X),\]
then
\[-\lim_{X \to \infty} \sum_{\substack{2\le\on{N}(\mf{a})\le X\\\mf{a}\in D(K, S)}} \frac{\mu(\mf{a})}{\on{N}(\mf{a})} = \delta(S).\]
\end{lemma}
\begin{proof}
The approach is similar to that of \cite[Theorem~6]{Al77}. By hypothesis, \cref{lem:mobius-inversion}, switching summation, and Axer's theorem as stated in \cref{thm:axer}, we have
\begin{align*}
-c_K\delta(S)X&\sim\sum_{\on{N}(\mf{a})\le X} Q_S(\mf{a}) = -\sum_{\on{N}(\mf{a})\le X}\sum_{\mf{b}\supseteq\mf{a}}\mu(\mf{b})\one_D(\mf{b}) = -\sum_{\on{N}(\mf{b})\le X}\mu(\mf{b})\one_D(\mf{b})\left[\frac{X}{\on{N}(\mf{b})}\right]_K\\
&= -\sum_{\on{N}(\mf{b})\le X}\mu(\mf{b})\one_D(\mf{b})c_K\frac{X}{\on{N}(\mf{b})} + o(X) = -c_KX\sum_{\substack{\on{N}(\mf{b})\le X\\\mf{b}\in D(K,S)}}\frac{\mu(\mf{b})}{\on{N}(\mf{b})} + o(X),
\end{align*}
as desired. The hypotheses of Axer's theorem, which we have applied to the function $\mf{a}\mapsto\mu(\mf{a})\one_D(\mf{a})$, are satisfied because of the classical bound
\[\sum_{\on{N}(\mf{a})\le X}1 = [x]_K = c_KX + o(X) = O(X)\]
and \cref{lem:mu-D-sum}, respectively.
\end{proof}
Now we are ready to prove \cref{thm:main-result}.
\begin{proof}[Proof of \cref{thm:main-result}]
The desired follows immediately from \cref{thm:key-step} and \cref{lem:convert-thm}.
\end{proof}

\section{Applications}\label{sec:applications}
Using \cref{thm:main-result}, we can reconstruct the results of Alladi \cite{Al77}, Dawsey \cite{Da17}, and Sweeting and Woo \cite{SW19} by choosing $S$ appropriately. For example, the formula in \cite{SW19} is recovered by letting $S$ denote the set of primes within a certain Chebotarev class in $L/K$ and using the effective Chebotarev density result of \cite{LO77}. However, our more general framework allows for extensions to other special sets of primes. The next two applications are elementary and of separate interest.
\begin{example}
If $S$ a finite set of primes, then
\[-\sum_{\substack{n \ge 2 \\ p_{\min}(n) \in S}} \frac{\mu(n)}{n} = 0,\]
and if $S$ is a cofinite set of primes, then
\[-\sum_{\substack{n \ge 2 \\ p_{\min}(n) \in S}} \frac{\mu(n)}{n} = 1.\]
\end{example}
\begin{example}
We can obtain a formula for densities of primes lying within Beatty sequences due to the work of Banks and Shparlinski \cite{BS09}. An irrational number $\alpha$ is said to be of finite type if there exists $N > 0$ such that there are only finitely many relatively prime pairs $(p, q)$ with
\[\left|\alpha - \frac{p}{q}\right|\le\frac{1}{q^N}.\]
Given such $\alpha > 1$, let $\mc{B}_\alpha$ denote the set of primes of the form $\lfloor\alpha n\rfloor$ for some $n\in\NN$. Then the prime number theorem for Beatty sequences \cite[Corollary~5.5]{BS09} implies $\delta(\mc{B}_{\alpha}) = 1/\alpha$. Thus
\[-\sum_{\substack{n\ge 2\\p_{\min}(n)\in\mc{B}_\alpha}}\frac{\mu(n)}{n} = \frac{1}{\alpha}.\]
As a special case, we obtain the following formula for $\pi$, since $\pi$ is of finite type \cite{Sa08}. Consider the Beatty sequence $\lfloor\pi n\rfloor$. We have
\[
-\sum_{\substack{n\ge 2\\ p_{\min}(n)\in\mc{B}_{\pi}}}\frac{\mu(n)}{n} = \frac{1}{\pi}.
\]
However, the series converges rather slowly. Let
\[S(X) = -\sum_{\substack{2\le n\le X\\p_{\min}(n)\in\mc{B}_\pi}}\frac{\mu(n)}{n}.\]
We have the following partial sums:
\begin{center}
\begin{tabular}{|c|r|}\hline
   $X$  &  $S(X)\hspace{.3in}$\\[2pt]\Xhline{1.2pt}
   10& 0.33333...\\\hline
   100& 0.23915...\\\hline
   1000& 0.31849...\\\hline
   10000& 0.34409...\\\hline
   100000& 0.34209...\\\hline
   1000000& 0.33181...\\\hline
   10000000& 0.32456...\\\hline
   100000000& 0.32117...\\\Xhline{1.2pt}
   $\infty$ &$1/\pi=$ 0.31831...\\\hline
\end{tabular}
\end{center}
\end{example}
\begin{example}
We can in fact intersect sets of Beatty primes and Chebotarev primes. Fix an irrational number $\alpha > 1$ of finite type, a finite Galois extension $K/\QQ$ with Galois group $G$, and a conjugacy class $C$ of $G$. Let $S$ denote the set of primes of the form $\lfloor\alpha n\rfloor$ for some $n\in\NN$ such that it is unramified in $K/\QQ$ and has Artin symbol $C$. Then by the work of \cite{JKM19}, we have
\[-\sum_{\substack{n\ge 2\\p_{\min}(n)\in S}}\frac{\mu(n)}{n} = \frac{1}{\alpha}\cdot\frac{\#C}{\#G}.\]
\end{example}
\begin{example}
Let $E$ be an elliptic curve over $\QQ$. For a prime $p$, let $a_E(p) := p + 1 - \# E(\FF_p)$ denote the trace of Frobenius.
We can obtain a formula for the density of Lang-Trotter primes, i.e., primes for which $a_E(p)$ is a fixed constant. For any $a$, we have
\[-\sum_{\substack{n\ge 2\\a_E(p_{\min}(n)) = a}} \frac{\mu(n)}{n} = 0\]
if $E$ does not have CM (see \cite{Se81}, for example). This also holds if $E$ has CM and $a \neq 0$, as then the primes lie in a fixed quadratic progression \cite{De41}, which has density $0$. When $E$ has CM and $a = 0$, the density of such primes is $1/2$ by \cite{De41}, so we have
\[-\sum_{\substack{n\ge 2\\a_E(p_{\min}(n)) = 0}} \frac{\mu(n)}{n} = \frac{1}{2}.\]
\end{example}
\begin{example}
Now we switch to studying the distribution of $\theta_p$, chosen such that $\theta_p \in [0,\pi]$ and $\cos\theta_p = a_E(p)/(2\sqrt{p})$. Suppose that $E$ does not have CM. We can consider the density of primes with $\theta_p$ lying in a subinterval of $[0,\pi]$.
From the work of Barnet-Lamb, Geraghty, Harris, and Taylor \cite{BGHT11} on the Sato-Tate conjecture, we have
\[\lim_{x\to\infty}\frac{1}{\pi(x)}\sum_{\substack{p\le x\\\theta_p\in [\alpha_1,\alpha_2]}}1 = \frac{2}{\pi}\int_{\alpha_1}^{\alpha_2}\sin^2\theta\,d\theta.\]
We deduce the following formula for Sato-Tate primes:
\[-\sum_{\substack{n\ge 2\\\theta_{p_{\min}(n)} \in [\alpha_1, \alpha_2]}} \frac{\mu(n)}{n} = \frac{2}{\pi}\int_{\alpha_1}^{\alpha_2}\sin^2\theta\,d\theta.\]
As a numerical example, consider the elliptic curve $E$ with Weierstrass equation $y^2 = x^3-x+1$, which does not have CM. Note that the Sato-Tate measure of $[\pi/3, 2\pi/3]$ is approximately $0.60900$. We calculate
\[
-\sum_{\substack{2\le n\le 1000000\\ \theta_{p_{\min}(n)}\in [\pi/3, 2\pi/3]}}\frac{\mu(n)}{n}\approx 0.60805.
\]
\end{example}

\section*{Acknowledgements}\label{sec:acknowledgements}
This research was conducted at the Emory University Mathematics Research Experience for Undergraduates, supported by the NSA (grant H98230-19-1-0013) and the NSF (grants 1849959, 1557960). We thank Harvard University and the Asa Griggs Candler Fund for their support. The authors thank Professor Ken Ono and Jesse Thorner for guidance and suggestions. We also thank Peter Humphries for pointing out some additional references.

\bibliographystyle{amsplain_mod2}
\bibliography{main}

\end{document}